\documentclass[11pt]{amsart}

 \usepackage{amsfonts,graphics,amsmath,amsthm,amsfonts,amscd,amssymb,amsmath,latexsym,multicol,
 mathrsfs}
\usepackage{epsfig,url}
\usepackage{flafter}
\usepackage{fancyhdr}
\usepackage{hyperref}
\hypersetup{colorlinks=true, linkcolor=black}

 \usepackage[matrix, arrow]{xy}

\DeclareMathOperator{\lct}{lct}

\DeclareMathOperator{\Supp}{Supp}

\DeclareMathOperator{\vol}{vol}

 \numberwithin{equation}{subsection}
 \numberwithin{footnote}{subsection}

 \newtheorem{cor}[subsection]{Corollary}

 \newtheorem{thm}[subsection]{Theorem}
 \newtheorem{conj}[subsection]{Conjecture}
\newtheorem{prob}[subsection]{Problem}

 \newtheorem{quest}[subsection]{Question}

{
\theoremstyle{definition}
 
 \newtheorem{exa}[subsection]{Example}

}

 \newcommand{\N}{\mathbb N}
 \newcommand{\PP}{\mathbb P}
 \newcommand{\A}{\mathbb A}
 \newcommand{\Q}{\mathbb Q}
 \newcommand{\R}{\mathbb R}
 \newcommand{\Z}{\mathbb Z}
  
 \newcommand{\bir}{\dashrightarrow}
 \newcommand{\rddown}[1]{\left\lfloor{#1}\right\rfloor} 

\title{\large P\MakeLowercase{ositivity, singularities, and boundedness}}
\author{\large C\MakeLowercase{aucher} B\MakeLowercase{irkar}}
 
\date{\today}
\begin{document}
\maketitle

\text{\hspace{2.5cm} In honour of James M$^{\rm c}$Kernan's 60's birthday.}

\begin{abstract}
In this short note we will explore some recent connections between positivity, singularities, and boundedness in various contexts focusing on birational geometry.  
\end{abstract}

\tableofcontents


\section{Introduction}

We will work over an algebraically closed field $k$ of characteristic zero. We assume familiarity with basic of birational geometry, c.f. [\ref{BCHM}].

In algebraic geometry, arithmetic geometry, mathematical physics and other fields, it is often a fundamental problem to understand how positivity and singularities are related. Positivity is usually of a global nature while singularity is usually of a local nature. Their relation is sometimes straightforward but other times very subtle. On the other hand, boundedness of families of varieties or certain invariants is also closely related to both positivity and singularities. 

We start with some examples.

\begin{exa}[Multiplicity]
Assume $X\subset \PP^d$ is a hypersurface of degree $r$. What is the largest possible multiplicity $\mu_xX$ at a closed point $x\in X$? Pick a line $L$ through $x$ not contained in $X$. Then $\mu_xX\le X\cdot L=r$. So we get an upper bound for the multiplicity which is a local measure of singularities, in terms of the degree which is a global measure of positivity.
\end{exa}

\begin{exa}[Ramification]
Let $\pi\colon X\to \PP^1$ be a finite morphism from a smooth projective curve of genus $g$. Then the Riemann-Hurwitz formula says 
$$
K_X=\pi^*K_{\PP^1}+R
$$ 
where $R$ is the ramification divisor. So 
$$
2g-2=\deg K_X=-2\deg\pi+\deg R.
$$
This relates the genus of $g$ which is a global measure of positivity to degree of the ramification divisor which is a measure of singularities of the morphism $\pi$. 
\end{exa}

These are simple examples. Here are some more subtle examples. 

\begin{thm}[{Singularities of linear systems, [\ref{B-BAB}]}]\label{t-sing-lin-sys}
Let $d,r\in \mathbb{N}$ and $\epsilon\in \mathbb{R}^{>0}$. 
Then  there is $t\in \mathbb{R}^{>0}$ depending only on $d,r,\epsilon$ satisfying the following. 
Assume 
\begin{itemize}

\item  $(X,B)$ is a projective $\epsilon$-lc pair of dimension $d$, 

\item $A$ is a very ample divisor on $X$ with $A^d\le r$,

\item $A-B$ is pseudo-effective, and 

\item $M\ge 0$ is an $\R$-Cartier $\R$-divisor with $A-M$ pseudo-effective.
\end{itemize}
Then the lc threshold  
$$
\lct(X,B,M)\ge t.
$$
\end{thm}

Here $A^d\le r$ means that $X$ belongs to a bounded family of varieties. The condition $A-M$ pseudo-effective controls the global positivity of $M$ while the conclusion of the theorem controls the singularities of $M$.

\begin{thm}[{BAB, [\ref{B-BAB}]}]\label{t-BAB}
Let $d\in \mathbb{N}$ and $\epsilon \in \mathbb{R}^{>0}$. Then the projective 
varieties $X$ such that  
\begin{itemize}

\item $(X,B)$ is $\epsilon$-lc of dimension $d$ for some boundary $B$, and 

\item  $-(K_X+B)$ is nef and big,
\end{itemize}
form a bounded family. 
\end{thm}

This theorem also relates singularities and positivity as it is evident from the statement. 
Examining the proof closely shows that its proof relies crucially on the connection between singularities and positivity including Theorem \ref{t-sing-lin-sys} on singularities of linear systems and other subtle connections between volumes and singularities of divisors.

\subsection{Singularities on fibrations}

In algebraic geometry and other fields it is often very important to understand singularities on ``fibrations". Here by fibration we simply mean a surjective morphism from a variety to another variety, usually assumed to be projective and sometimes with connected fibres. 

\begin{prob}
Suppose $(X,B)$ is a pair, $A$ is a divisor on $X$, and $f\colon X\to Z$ is a surjective morphism. Relate the singularities of $(X,B)$, of the fibres of $f$, and of $Z$, taking into account positivity of $A$. 
\end{prob}

This is of course stated vaguely because it is in a very general context but it can serve as a good guiding problem. Considering various special cases, we can appreciate the importance and difficulty of this problem. For example, we can view Theorem \ref{t-sing-lin-sys} as a case of this problem (indeed, it is enough to consider only those $M\sim_\R A$, so the theorem is about singularities of the real linear system $|A|_\R$).

In what follows we will explore more cases of the problem. 

\subsection{Singularities on Fano type fibrations}

Fano varieties and more generally Fano fibrations are extremely important in birational geometry. One can view the minimal model program (MMP) as looking for Fano fibrations on algebraic varieties. Here, by Fano fibration we mean a contraction $X\to Z$ where $-K_X$ is ample over $Z$ ($f$ being a contraction means $f$ is projective and $f_*\mathcal{O}_X=\mathcal{O}_Z$). In particular, we also include the cases when $X\to Z$ is birational or even identity. So we can view many ingredients of the MMP as Fano fibrations, e.g. Mori fibre spaces, divisorial and flipping contractions, singularities. In the context of Fano fibrations, positivity is measured by the anti-canonical divisor $-K_X$.

There are three main types of Fano fibrations. First, we have the global case when $\dim Z=0$. Some of the main problems in this case were settled in [\ref{B-compl}][\ref{B-BAB}]. Second, we have the genuine fibre case when $\dim X>\dim Z>0$. Some of the main problems in this case were established in [\ref{B-Fano-fib}][\ref{B-FT-fib}][\ref{B16}]. Third, we have the birational (that is local) case when $\dim X=\dim Z$. There are still many problems in this direction some of which will be mentioned below.

\begin{thm}[{Singularities on Fano type fibrations, [\ref{B-Fano-fib}]}]
	\label{t-mc-sh-conj}
	Let $d\in \mathbb{N}$ and let $\epsilon \in \mathbb{R}^{>0}$. Then there is $\delta \in \mathbb{R}^{>0}$ depending only on $d,\epsilon$ satisfying the following. 
	Let $(X,B)$ be a pair where $B$ is a $\Q$-boundary and let \(f\colon X\to Z\) be a contraction 
	such that 
\begin{itemize}
\item $(X,B)$ is $\epsilon$-lc and $\dim X-\dim Z=d$,
\item $K_X+B\sim_\Q 0/Z$, and 
\item $-K_X$ is big over $Z$ (i.e. $X$ is of Fano type over $Z$).
\end{itemize}	
Then the generalised pair $(Z,B_Z+M_Z)$ induced by the canonical bundle formula 
$$
K_X+B\sim_\Q f^*(K_Z+B_Z+M_Z)
$$ 
is generalised $\delta$-lc.
\end{thm}

This was conjectured by Shokurov. And here are two quick consequences. 

\begin{cor}
	\label{cor-mckernan-conj}
	Let $d\in \mathbb{N}$ and let $\epsilon \in \mathbb{R}^{>0}$. Then there is $\delta \in \mathbb{R}^{>0}$ depending only on $d,\epsilon$ satisfying the following. 
	Let \(f\colon X\to Z\) be a Fano fibration where $X$ is $\epsilon$-lc, $\dim X-\dim Z=d$, and $K_Z$ is $\Q$-Cartier. Then $Z$ is $\delta$-lc.
\end{cor}

This was conjectured by M$^{\rm c}$Kernan.

\begin{cor}
	\label{cor-mult-fib}
	Let $d\in \mathbb{N}$ and let $\epsilon \in \mathbb{R}^{>0}$. Then there is $l\in \mathbb{N}$ depending only on $d,\epsilon$ satisfying the following. 
	Let \(f\colon X\to Z\) be a Fano fibration where $X$ is $\epsilon$-lc of dimension $d$ and $Z$ is a  curve.
	Then for each $z\in Z$, each coefficient of $f^*z$ is $\le l$.
\end{cor}

These results are used to prove the following result which was conjectured by M$^{\rm c}$Kernan and Prokhorov (in a stronger form).

\begin{thm}[{Boundedness of rationally-connected varieties with nef $-K$, [\ref{B-Fano-fib}]}]
	\label{t-mc-pro-conj}
	Let $d\in \mathbb{N}$ and let $\epsilon \in \mathbb{R}^{>0}$. Consider projective varieties $X$  
    where 
	\begin{itemize}
	\item $(X,B)$ is $\epsilon$-lc of dimension $d$ for some $B$,
	\item $-(K_X+B)$ is nef, and 
	\item $X$ is rationally connected.
	\end{itemize}
	Then such $X$ form a bounded family up to isomorphism in codimension one.
\end{thm}

Here is a very rough plan of the proof of Theorem \ref{t-mc-sh-conj} which overlaps with the proof of Theorem \ref{t-bnd-mult-lc-places-fib-main} below. 
\vspace{0.2cm}
\begin{itemize}
\item Step 1: Fix a closed point $z\in Z$. Apply \emph{complement theory}, to get a lc $n$-complements $K_X+\Lambda$ over $z$ where $n$ is bounded.
\vspace{0.2cm}

\item Step 2: Apply \emph{boundedness of Fano's} (BAB) to the general fibres and get a birational model $Y\to Z$ so that if $D$ is the birational transform of the horizontal part of $\Lambda$, then $(Y,D)$ is relatively bounded over $Z$.
\vspace{0.2cm}

\item Step 3: Apply the techniques of families of \emph{nodal curves} to get a relatively bounded toroidal model $(Y',D')\to (Z',E')$ of $(Y,D)\to Z$. 
\vspace{0.2cm}

\item Step 4: Find a local \emph{toric} model $(Y'',D'')$ of $(Y',D')$ near $y'$ where $y'$ is the centre of an appropriate lc place of $(X,\Lambda)$.
\vspace{0.2cm} 

\item Step 5: Translate the conclusion of the theorem into a toric problem using the toric model $(Y'',D'')$. This toric problem happens to be Theorem \ref{t-toric-problem} below. 
\vspace{0.2cm}

\item Step 6: Solve the toric problem.
\vspace{0.2cm}
\end{itemize}

\bigskip

\subsection{Singularities on Calabi-Yau fibrations}

One can ask whether Theorem \ref{t-mc-sh-conj} holds if we remove the Fano type property, that is, removing $-K_X$ being big. The answer is no. Indeed, elliptic fibrations on surfaces already provide a counter-example because it is well-known that there is no upper bound on the multiplicity of fibres of such fibrations. But that is not the end of the story. In order to control singularities, we need to control positivity by introducing a ``polarising" divisor.  

\begin{thm}[{Singularities on polarised Calabi-Yau fibrations, [\ref{B-Fano-fib}]}]\label{t-cb-sing-usual-fib}
	Let $d,v\in \mathbb{N}$ and let $\epsilon \in \mathbb{R}^{>0}$.
 Then there is $\delta\in \R^{>0}$ depending only on $d,v,\epsilon$ satisfying the following. 
 Assume that $(X,B)$ is a pair where $B$ is a $\Q$-boundary, and that $f\colon X\to Z$ is a contraction such that
\begin{itemize}
\item $(X,B)$ is $\epsilon$-lc of dimension $d$,

\item $K_X+B\sim_\Q 0/Z$, and  

\item there is an integral divisor $A$ on $X$ which is big over $Z$ with $\vol(A|_F)\le v$ 
for the general fibres $F$ of $f$.
\end{itemize}
Then the generalised pair $(Z,B_Z+M_Z)$ induced by the canonical bundle formula is generalised $\delta$-lc.
\end{thm} 

Similar to the Fano-type case, this immediately implies boundedness of multiplicities of fibres over codimension one points and boundedness of singularities of $Z$ in case $K_Z$ is $\Q$-Cartier. 
The theorem is very useful for studying Calabi-Yau fibrations (e.g. in moduli theory [\ref{B-moduli}]) when used in conjunction with [\ref{B-geom-pol-var}][\ref{BZh}][\ref{B-bnd-vol-gen-pairs}].

\bigskip

\subsection{Singularities on more general fibrations}

In order to prove the above results on singularities on Fano type and Calabi-Yau fibrations, it is important to consider more general fibrations as in the next result.

\begin{thm}[{[\ref{B-Fano-fib}]}]\label{t-bnd-mult-lc-places-fib-main}
	Let $d,r\in \mathbb{N}$ and let $\epsilon \in \mathbb{R}^{>0}$.
Then there is $l\in \N$ depending only on $d,r,\epsilon$ satisfying the 
following. Assume that 
\begin{itemize}

\item $(W,B_W)$ is an $\epsilon$-lc pair and $(X,\Lambda)$ is an lc pair, both of dimension $d$,

\item $W\to X$ is a birational contraction and $X\to Z$ is a surjective projective morphism onto a smooth curve, 

\item $K_W+B_W$ is nef$/X$ (or nef$/Z$), 

\item  $A$ is a very ample$/Z$ divisor on $X$ such that $\deg_{A/Z}A\le r$,

\item $A-B$ and $A-\Lambda$ are pseudo-effective over $Z$ where $B$ is the pushdown of $B_W$, 

\item $r\Lambda$ is integral, and 

\item $T$ is a prime divisor over $X$ mapping to a closed point $z\in Z$ with 
$$
a(T,W,B_W)\le 1 ~~~\mbox{and}~~~ a(T,X,\Lambda)=0.
$$
\end{itemize}
Then $\mu_TF\le l$ where $F$ is the fibre of $X\to Z$ over $z$. That is, if $\phi\colon V\to X$ is a 
resolution so that $T$ is a divisor on $V$, then $\mu_T\phi^*F\le l$.
\end{thm}

This looks technical but in essence it says that on fibrations $(W,B_W)\to Z$ where $(W,B_W)$ has good singularities, $K_W+B_W$ is nef over $Z$ (a minimality condition), and $(W,B_W)$ is ``generically bounded" over $Z$ (in some sense), then we can control singularities. Similar conditions also appear in Theorems \ref{t-mc-sh-conj} and \ref{t-cb-sing-usual-fib}.

We now state a simpler variant of \ref{t-bnd-mult-lc-places-fib-main} which is easier to digest and also much easier to prove. 
 
\begin{thm}\label{t-lct-gen-type}
	Let $d,v\in \mathbb{N}$ and $\epsilon \in \mathbb{R}^{>0}$, and let $\Phi\subset \mathbb{Q}^{>0}$ be a DCC set. Then there is $t\in \mathbb{R}^{>0}$ depending only on $d,v,\epsilon, \Phi$ satisfying the 
following. Assume that 
\begin{itemize}

\item $(W,B_W)$ is an $\epsilon$-lc pair of dimension $d$,

\item $f\colon W\to Z$ is a contraction onto a smooth curve, 

\item the horizontal$/Z$ coefficients of $B_W$ belong to $\Phi$, and  

\item $K_W+B_W$ is nef and big$/Z$ with 
$$
\vol((K_W+B_W)|_F)\le v
$$
for the general fibres $F$ of $f$. 
\end{itemize}
Then $(W,B_W+tF)$ is lc for any fibre $F$ of $f$ over a closed point.
\end{thm}
 
It is possible to formulate similar results when the base $Z$ is not a curve. 

Even the case of the theorem when $W$ is a smooth surface and $B_W=0$ is informative. In this case, it is easy to control singularities of curve $C$ in the fibre $F$ for which $K_X\cdot C>0$. But if $K_X\cdot C=0$, then more work is required.

Since the theorem has not been published anywhere, we include a slightly sketchy proof for convenience, following [\ref{B16}]. 

\begin{proof}[Proof of Theorem \ref{t-lct-gen-type}] 
We will assume that $d\ge 2$ as the case $d=1$ is trivial.\\

	\emph{Step 1.}
	We can assume $K_W+B_W$ is ample over $Z$. 
	Fix a closed point \(z \in Z\).
	Take a log resolution \(\varphi \colon V\longrightarrow X\) of \((W,B_W + f^{*}z)\).
	Let 
\begin{itemize}
	\item \(\{M_{i}\}\) be the components of \(\varphi^{*}f^{*}z\),
	\item \(\{M'_j\}\) be the exceptional divisors of \(\varphi\) that are not among \(M_{i}\).
	\end{itemize}	
	Define
	\[
		\Gamma_V = \sum M_{i} + \sum (1 - {\epsilon}) M'_j + {B^h_W}^\sim
	\]
	where ${B^h_W}^\sim$ is the birational transform of the horizontal part $B^h_W$.
	Then 
\begin{itemize}
	\item \((V,\Gamma_V)\) is log smooth and dlt, 
	\item \(\rddown{\Gamma_V} = \Supp \varphi^{*}f^{*}z\),
	\item \(K_V + \Gamma_V\) is big over \(Z\).\\
\end{itemize}

	\emph{Step 2.}
	It is possible to run an MMP on \(K_V + \Gamma_V\) over $Z$ which ends with a minimal model
	\((Y,\Gamma_Y)\) where \(K_Y+\Gamma_Y\) is semi-ample over \(Z\).
	The main point is that \(\rddown{\Gamma_V} = \Supp \varphi^{*}f^{*}z\), so
	\((V, \Gamma_V - s\varphi^{*}f^{*}z)\) is klt for some small \(s>0\), and over $Z$ we have
	\[
		K_V + \Gamma_V - s\varphi^{*}f^{*}z \equiv K_V + \Gamma_V.
	\]\\

	\emph{Step 3.}
	Replace $Y$ with the ample model of $K_Y+\Gamma_Y$ over $Z$ so that we can assume $K_Y+\Gamma_Y$ is ample over $Z$. 
	One can check that if \(H\) is a general fibre of \(h\colon Y\longrightarrow Z\),
	then
	\[
		\vol((K_Y + \Gamma_Y)|_H)
	\]
	is bounded. Indeed, $W\bir Y$ is an isomorphism over the generic point of $Z$.
	
	Pick a component \(T\) of \(h^{*}z\).
	Define
	\[
		K_{T^\nu} + \Gamma_{T^\nu} = (K_Y + \Gamma_Y)|_{T^\nu}
	\]
	by adjunction, where ${T^\nu}$ is the normalisation of $T$.
	Since the coefficients of \(\Gamma_Y\) are in \(\Phi\cup \{1\}\), the
	coefficients of \(\Gamma_{T^\nu}\) are in some fixed DCC set.
	
	Now by [\ref{HMX-acc}],
	\[
		\vol(K_{T^\nu} + \Gamma_{T^\nu}) \geq \theta >0
	\]
	where \(\theta\) is fixed.\\

	\emph{Step 4.}
	Write \(h^{*}z = \sum m_p T_p\).
	Then
	\begin{align}
		\vol(K_H + \Gamma_H) & = (K_H + \Gamma_H)^{d-1}                   \\
		                     & = ((K_Y + \Gamma_Y)|_H)^{d - 1}            \\
		                     & = (K_Y + \Gamma_Y)^{d-1} \cdot H           \\
		                     & = (K_Y + \Gamma_Y)^{d - 1} \cdot h^{*}z    \\
		                     & = \sum m_p ((K_Y + \Gamma_Y)|_{T^\nu_p})^{d-1} \\
		                     & = \sum m_p \vol(K_{T^\nu_p} + \Gamma_{T^\nu_p})
	\end{align}
	where
	\[
		K_{T^\nu_p} + \Gamma_{T^\nu_p} = (K_Y + \Gamma_Y)|_{T^\nu_p}.
	\]
	Therefore,
	\[
		\vol(K_H + \Gamma_H) = \sum m_p \vol(K_{T^\nu_p} + \Gamma_{T^\nu_p}) \geq \sum m_p \theta.
	\]
	Since \(\vol(K_H + \Gamma_H)\) is bounded from above, we deduce that \(m_p\) are all bounded.\\

	\emph{Step 5.}
			We have then shown that the fibre of $Y\to Z$ over $z$ has bounded coefficients. 
	Write \(K_V + B_V = \varphi^{*}(K_W + B_W)\).
	Since \((W,B_W)\) is \(\epsilon\)-lc, we can see that 
	\[B_V + \epsilon \rddown{\Gamma_V} \leq \Gamma_V\] 
	and pushing this down to $Y$, we have 
	\[B_Y + \epsilon \rddown{\Gamma_Y} \leq \Gamma_Y.\]
	Thus by the previous step, there exists a fixed \(t > 0\) such that
	\[
		K_Y + B_Y + t h^{*}z \leq K_Y + \Gamma_Y.
	\]
	
	Now since \(K_W + B_W\) is nef over $Z$ and \(h^{*}z \equiv 0 / Z\), applying the negativity lemma on a common resolution of $W,Y$ we can see that the pullback of  
	\(K_W + B_W + t f^{*}z\) is at most the pullback of \(K_Y + \Gamma_Y\).
	Therefore, \((W, B_W + t f^{*}z)\) is lc.
\end{proof} 

\subsection{Toric singularities}

The following theorem was stated as a conjecture by the author in a seminar in 2018 in London. 

\begin{thm}[{Multiplicities on weighted blowups, [\ref{SS-weighted}]}]\label{t-SS}
Let $d\in \mathbb{N}$ and $\epsilon \in \mathbb{R}^{>0}$. Let $f\colon X\to \A^d$ be the weighted blowup given by the weights $(n_1,\dots,n_d)$ and assume that $X$ is $\epsilon$-lc. Then $\min\{n_1,\dots,n_d\}$ is bounded from above.
\end{thm}

The proof uses topological properties of subgroups of $\R^d$ disjoint from an open set. For now there is no known direct proof that works in every dimension and does not rely on geometry. 

Let $t_1,\dots,t_d$ be the coordinates on $\mathbb{A}^d$ and $H_i$ the vanishing set of $t_i$. 
Let $T$ be the exceptional divisor of $f$. Then $\mu_Tf^*H_i=n_i$. So the theorem says that the multiplicity $\mu_Tf^*H_i$ is bounded from above, for some $i$. We can also think of this as saying that the lc threshold $\lct(X,0,f^*H_i)$ is bounded from below for some $i$.


On the other hand, in the fibre case we have: 

\begin{thm}[{Multiplicities on toric Fano fibrations, [\ref{BC-toric}]}]\label{t-BC-toric}
Let $d\in \mathbb{N}$ and $\epsilon \in \mathbb{R}^{>0}$. Let $f\colon X\to Z$ be a toric Fano fibration where $X$ is $\epsilon$-lc. Then multiplicities of the fibres of $f$ over codimension one points of $Z$ are bounded from above.
\end{thm}

We will state a far more subtle generalisation of the previous theorem. 
Fix $d,r\in \mathbb{N}$ and $\epsilon\in \R^{>0}$. 
Pick a primitive vector $(n_1,\dots,n_d)$ of integers where $n_1\ge 0$. 
Let $e_1,\dots, e_d$ be the standard basis of $\Z^d$. The vectors 
$$
e_2,\dots,e_d, \ -(e_2+\dots+e_d), \ (n_1,\dots,n_d)
$$ 
generate a fan structure in $\R^d$. 
This defines a toric morphism $f\colon X\to \mathbb{A}^1$.
The reduced fibre $T=f^{-1}\{0\}$ is a prime divisor corresponding to $(n_1,\dots,n_d)$.
Take a (e.g. toric) resolution of singularities $\psi\colon W\to X$ with exceptional divisors $E_1,\dots,E_p$. 

\begin{thm}[{Multiplicities of fibres of toric-related fibrations, [\ref{B-Fano-fib}]}]\label{t-toric-problem} 
Given $d,r\in \mathbb{N}$ and $\epsilon\in \R^{>0}$ as above, assume that  
$$
T \ \not\subseteq \ {\bf B}(K_W+\sum_1^p(1-\epsilon)E_i-\psi^*rK_X).
$$ 
Then $n_1$ is bounded from above.
\end{thm}

Here ${\bf B}$ denotes stable base locus. The role of the term $-\psi^*rK_X$ is to control degree over $\A^1$. The theorem is a key ingredient of the results of [\ref{B-Fano-fib}].

In order to prove the theorem, we first attempted to prove a more general local variant of it. More precisely, 
suppose that $(W,B_W)$ is an $\epsilon$-lc pair of dimension $d$ (as usual $\epsilon>0$), $W\to \A^d$ is a birational contraction,  
and that $\mu_0B\le r$ where $B$ is the pushdown of $B_W$. Is it true that for any prime divisor $T$ on $W$ mapping to $0$, $\mu_TH_i$ is bounded from above for some $i$, where $H_i$ are the coordinate hyperplanes on $\A^d$? Although one can show that this is true for $d=2$ but it turns out that it is not true for $d\ge 3$. Here is a counter-example. 

\begin{exa}
Consider the vector $(n,n+1,n(n+1))$ and let $f\colon X\to \A^3$ be the corresponding weighted blowup  with exceptional divisor $T$. Let $H_1,\cdots, H_3$ be the coordinate hyperplanes on $\A^3$. Then on $X$ we have 
$$
f^*H_1=H_1^\sim+nT, \ f^*H_2=H_2^\sim+(n+1)T, \ f^*H_3=H_3^\sim+n(n+1)T
$$
where $H_i^\sim$ denotes birational transform.
From this we get 
$$
H_3^\sim\sim (n+1)H_1^\sim \ \mbox{and} \  H_3^\sim\sim nH_2^\sim
$$
over $\A^3$.
Therefore, $|H_3^\sim|$ is base point free which in particular means that $H_3^\sim$ is Cartier (and it is ample over $\A^3$). Let $r=6$. Pick a general $0\le B_X\sim_\Q rH_3^\sim$. 

 Let $W'\to X$ be a resolution of singularities. Let $B_{W'}$ be the pullback of $B_X$. Run MMP on $(W',B_{W'})$ over $X$ to get a minimal model $(W,B_W)$. By our choice of $r$ and by boundedness of length of extremal rays, $K_W+B_W$ is actually nef over $\A^3$. Moreover, $(W,B_W)$ is $\frac{1}{2}$-lc.
 
 Now let $g\colon V\to \A^3$ be the blowup of the origin, with exceptional divisor $E$, and denote $V\bir X$ by $\phi$. Then from $B_X\sim_\Q rH_3^\sim$ on $X$, we get $\phi^* B_X\sim_\Q \phi^* rH_3^\sim$ on $V$. Since $B_X$ is general, $\phi^*B_X$ is the birational transform of $B_X$. On the other hand, $g^*H_3=\phi^*H_3^\sim+eE$ where $e\le 1$. 

Let $B$ be the pushdown of $B_X$. Then $\mu_0B=(\phi^*B_X)\cdot L$ for some line $L$ in the exceptional divisor of $V\to \A^3$. From 
$$
\phi^* B_X\sim_\Q r\phi^*H_3^\sim \ \mbox{and} \ r\phi^*H_3^\sim+reE\equiv 0/\A^3,
$$ 
we deduce that 
$$
\mu_0B=(\phi^*B_X)\cdot L=r\phi^*H_3^\sim\cdot L=-reE\cdot L=re\le r.
$$ 
However, $\mu_TH_i\ge n$ where $n$ is arbitrary. Note that by construction, $T$ is not contracted over $W$, so it its birational transform is a prime divisor on $W$. 
\end{exa} 

\subsection{Geometry of numbers}

It turns out that the results stated in the toric section above have consequences for the geometry of numbers. We can translate those results into the language of integers which would be understandable to advanced high school students. 

Define
$$
\mathcal{M}_d=\{(m_1, \dots, m_d) \mid m_i\in \Z^{\ge 0}, \ \gcd(m_1,\dots,m_d)=1\}.
$$ 
By convention, $(0,\dots,0)\notin \mathcal{M}_d$. 
Fix $(n_1,\dots, n_d)\in \mathcal{M}_d$. 
Define 
$
\alpha \colon \mathcal{M}_d \to \Q
$
by writing 
$$
(m_1, \dots, m_d)=a(n_1,\dots, n_d)+\sum_{i\neq j} b_ie_i, \ \ \ \mbox{for some $1\le j\le d$},
$$ 
where $e_1, \dots, e_d$ is the standard basis of $\R^d$, and $a,b_i\ge 0$, 
and letting 
$$
\alpha(m_1,\dots, m_d)=a+\sum_{i\neq j} b_i.
$$

We can ask the question: how small can $\alpha$ get?

\begin{exa}
We list some examples.
\begin{itemize}
\item $(n_1,n_2)=(1,1) \ \implies  \ \alpha\ge 1$. 

\item $(n_1,n_2)=(1,2) \ \implies  \ \alpha\ge 1$. 

\item $(n_1,n_2)=(1,n) \ \implies  \ \alpha\ge 1$.  

\item $(n_1,n_2)=(2,3) \ \implies  \ \alpha\ge \frac{2}{3}$. 

\item $(n_1,n_2)=(n,n-1) \ \implies  \ \alpha\ge \frac{2}{n}$. 

\item $(n_1,n_2)=(n,2n-1) \ \implies  \ \alpha\ge \frac{2}{n}$. 
\end{itemize} 
\end{exa}

\begin{exa}
Here are some higher rank examples.
\begin{itemize}
\item $(n_1,n_2,n_3)=(1,1,1) \ \implies  \ \alpha\ge 1$. 

\item $(n_1,n_2,n_3)=(1,1,2) \ \implies  \ \alpha\ge 1$. 

\item $(n_1,n_2,n_3)=(1,n_2,n_3) \ \implies  \ \alpha\ge 1$. 

\item $(n_1,n_2,n_3)=(10,11,19) \ \implies  \ \alpha(1,1,2)=\frac{5}{11}$. 

\item $(n_1,\dots,n_4)=(20,57,133,210) \ \implies  \ \alpha\ge 1$. 

\item $(n_1,\dots,n_4)=(32,41,71,102) \ \implies  \ \alpha\ge 1$. 

\item In general, 
$$
\alpha\ge \frac{1}{\min\{n_1,\dots,n_d\}}.
$$ 
\end{itemize} 
\end{exa}
 
The two examples of dimension 4 were copied from [\ref{SS-weighted}, after Proposition 4.11]. 

Let $f\colon X\to \A^d$ be the weighted blowup given by the weights $(n_1,\dots,n_d)$. Then $X$ is $\epsilon$-lc iff the corresponding function $\alpha\ge \epsilon$. So the condition $\alpha\ge \epsilon$ implies that the minimum of the $n_i$ is bounded from above, by Theorem \ref{t-SS}. Thus we get a very elementary interpretation of Theorem \ref{t-SS}.

On the other hand, we can define another kind of function.
Let
$$
\mathcal{M}_d'=\{(m_1,\dots,m_d) \mid m_1\in \Z^{\ge 0}, m_i\in \Z, i\ge 2, \ \gcd(m_1,\dots,m_d)=1\}.
$$ 
By convention, $(0,\dots, 0)\notin \mathcal{M}_d'$. 
Fix $(n_1,\dots,n_d)\in \mathcal{M}_d'$. 
Define 
$
\alpha' \colon \mathcal{M}_d' \to \Q
$
by writing
$$
(m_1,\dots,m_d)=a(n_1,\dots,n_d)+\sum_{2\le i\le d} b_iu_i, \ \ \ a,b_i\ge 0,
$$ 
where 
$$
\{u_2,\dots,u_d\}\subset \{e_2,\dots,e_d, \ -(e_2+\dots+e_d)\}
$$
and letting 
$$
\alpha'(m_1,\dots,m_d)=a+\sum_{2\le i\le d} b_i.
$$

We can ask again, how small can $\alpha$ get?

\begin{exa}
We list some examples.
\begin{itemize}
\item $(n_1,n_2)=(1,1) \ \implies  \ \alpha\ge 1$. 

\item $(n_1,n_2)=(1,2) \ \implies  \ \alpha\ge 1$. 

\item $(n_1,n_2)=(1,n) \ \implies  \ \alpha\ge 1$. 

\item $(n_1,n_2)=(2,3) \ \implies  \ \alpha\ge 1$. 

\item $(n_1,n_2)=(n,1) \ \implies  \ \alpha\ge \frac{2}{n}$. 
\end{itemize} 
\end{exa}

Let $f\colon X\to \mathbb{A}^1$ be the morphism defined before \ref{t-BC-toric} using the vector $(n_1,\dots,n_d)$.
The reduced fibre $T=f^{-1}\{0\}$ is the prime divisor corresponding to $(n_1,\dots,n_d)$.
Then $X$ is $\epsilon$-lc iff $\alpha'\ge \epsilon$.
Theorem \ref{t-BC-toric} then says that if $\alpha'\ge \epsilon$, then $n_1=\mu_Tf^*0$ is bounded from above. Thus we get a very elementary interpretation of Theorem \ref{t-BC-toric}. 

What is interesting about these statements is that everything is built from the vector $(n_1,\dots,n_d)$. So the statements can be viewed as purely number-theoretic. 

It is possible to define more functions of the type above and get similar boundedness results. Moreover, Theorem \ref{t-toric-problem} can also be translated into a number-theoretic statement (although a long statement) because again everything in the statement is built from the given vector $(n_1,\dots,n_d)$.

\subsection{Some open problems}

There are many interesting open problems regarding positivity, singularities, and boundedness. Here we mention a few.

\begin{conj}[Birkar-Shokurov]
For $d\in \N, \epsilon \in \R^{>0}$, there is $t \in \R^{>0}$ such that if 
\begin{itemize}
\item $(X,B)$ is $\epsilon$-lc of dimension $d$, and 
\item $f\colon X\to Z$ is a contraction,
\item $-(K_X+B)$ is ample over $Z$,
\end{itemize}
then for each closed point $z\in Z$ there is a Cartier divisor $D\ge 0$ containing $z$ such that 
$$
(X,B+tf^*D) 
$$ 
is log canonical.
\end{conj}

This is closely related to:

\begin{conj}[Shokurov]
For $d\in \N, \epsilon \in \R^{>0}$, there is $n \in \N$ such that if 
\begin{itemize}
\item $X$ is $\epsilon$-lc of dimension $d$,  
\item $f\colon X\to Z$ is a contraction, and 
\item $-K_X$ is ample over $Z$,
\end{itemize}
then for each closed point $z\in Z$, there is a klt $n$-complement $K_X+B$ over $z$.
\end{conj}

Theorem \ref{t-bnd-mult-lc-places-fib-main} looks technical partially because it carries extra information (e.g. $\Lambda$) originating in complement theory which is used to reduce the theorem to a toric problem. What if one removes this extra structure? To be more precise: 

\begin{quest}
	Let $d,r\in \mathbb{N}$ and let $\epsilon \in \mathbb{R}^{>0}$.
Then is there $l\in \N$ depending only on $d,r,\epsilon$ satisfying the 
following? Assume that 
\begin{itemize}

\item $(W,B_W)$ is an $\epsilon$-lc pair of dimension $d$,

\item $W\to X$ is a birational contraction and $X\to Z$ is a contraction onto a smooth curve,

\item $K_W+B_W$ is nef$/X$ (or nef $/Z$), 

\item  $A$ is a very ample$/Z$ divisor on $X$ such that $\deg_{A/Z}A\le r$,

\item $A-B$ is pseudo-effective$/Z$ where $B$ is the pushdown of $B_W$, and 

\item $T$ is a prime divisor on $X$ mapping to a closed point $z\in Z$. 
\end{itemize}
Then $\mu_TF\le l$ where $F$ is the fibre of $X\to Z$ over $z$.
\end{quest} 

Here is another variant: 

\begin{quest}
	Let $d,r\in \mathbb{N}$ and let $\epsilon \in \mathbb{R}^{>0}$. Is the following true?
 Assume that 
\begin{itemize}

\item $(W,B_W)$ is a projective $\epsilon$-lc pair of dimension $d$,

\item $W\to X$ is a birational contraction,

\item $K_W+B_W$ is nef$/X$, 

\item  $A$ is a very ample divisor on $X$ such that $\deg_{A}A\le r$,

\item $A-B$ is pseudo-effective where $B$ is the pushdown of $B_W$. 
\end{itemize}
Then such $X$ form a bounded family.
\end{quest} 

In case $K_W+B_W$ is numerically trivial over $X$, then $(W,B_W)$ is a crepant model of $(X,B)$. Boundedness of such models was established in [\ref{B-FT-fib}].

\bigskip
\emph{Acknowledgements.} This work was supported by a grant of Tsinghua University
and a grant of the National Program of Overseas High Level Talent.


\vspace{2cm}

\small
\textsc{Yau Mathematical Sciences Center, JingZhai, Tsinghua University, Hai Dian District, Beijing, China 100084  } \endgraf
\vspace{0.5cm}
\email{Email: birkar@tsinghua.edu.cn\\}

\end{document}